\documentclass[a4paper,12pt]{amsart}

\usepackage[english]{babel}
\usepackage{amsthm}
\usepackage{amssymb}
\usepackage{hyperref}

\newcommand{\Set}[1]{\left\{\, #1 \,\right\}}

\newcommand{\Zero}[0]{\Set{0}}
\newcommand{\One}[0]{\Set{1}}

\newcommand{\Span}[1]{\langle\, #1 \,\rangle}

\newcommand{\Size}[1]{\left\lvert #1 \right\rvert}

\renewcommand{\phi}[0]{\varphi}
\renewcommand{\theta}[0]{\vartheta}
\renewcommand{\epsilon}[0]{\varepsilon}

\newcommand{\F}{\text{$\mathbf{F}$}}

\newcommand{\ismax}[0]{\lessdot}

\theoremstyle{plain}

\newtheorem{dummy}{Dummy}
\numberwithin{dummy}{section}
\numberwithin{equation}{section}

\newtheorem{theorem}[dummy]{Theorem}
\newtheorem{lemma}[dummy]{Lemma}

\newtheorem{proposition}[dummy]{Proposition}
\newtheorem{corollary}[dummy]{Corollary}
\newtheorem{conj}[dummy]{Conjecture}

\theoremstyle{definition}

\newtheorem{definition}[dummy]{Definition}

\theoremstyle{remark}

\newcommand{\FMEO}[1]{}

\begin{document}

\bibliographystyle{amsalpha}

\date{8 January 2015, 13:26 CET --- Version 7.07
}

\title[Finite morphic $p$-groups]
{Finite morphic $p$-groups}

\author{A.~Caranti}

\address[A.~Caranti]%
 {Dipartimento di Matematica\\
  Universit\`a degli Studi di Trento\\
  via Sommarive 14\\
  I-38123 Trento\\
  Italy} 

\email{andrea.caranti@unitn.it} 

\urladdr{http://science.unitn.it/$\sim$caranti/}

\author{C.~M.~Scoppola}

\address[C.~M.~Scoppola]
{Dipartimento di Ingegneria e Scienze dell'Informazione e Matematica\\
 Via Vetoio (Coppito 1)\\
 I-67100 Coppito\\
Italy}

\email{scoppola@univaq.it}

\begin{abstract}
 According to Li, Nicholson and Zan, a group $G$ is said to be morphic
 if, for  every pair $N_{1}, N_{2}$  of normal subgroups, each  of the
 conditions $G/N_{1}  \cong N_{2}$  and $G/N_{2} \cong  N_{1}$ implies
 the other.  Finite, homocyclic $p$-groups  are morphic, and so is the
 nonabelian group  of order $p^{3}$ and  exponent $p$, for $p$  an odd
 prime. It follows from results of An,  Ding and  Zhan on self dual
 groups that these  are the only  examples of
 finite, morphic $p$-groups. In this paper we obtain the same result
 under a weaker hypotesis.
\end{abstract}

\keywords{finite $p$-groups, morphic groups, self dual groups}

\thanks{The first author gratefully acknowledges the support of the
  Department of Mathematics, University of Trento. 
\endgraf    
The second author gratefully acknowledges the support of the
  Department of Mathematics, University of L'Aquila. 
\endgraf
The authors are
  members of GNSAGA---INdAM}

\subjclass[2010]{20D15}

\maketitle

\thispagestyle{empty}

\section{Introduction}

A module $M$ over  a ring
is  said  to be  \emph{morphic}  if  whenever  a  submodule $N$  is  a
homomorphic image  of $M$, so that  there is an epimorphism  $\phi : M
\to N$, then $M/N \cong \ker(\phi)$. In other words, if $N_{1}$ and
$N_{2}$ are submodules of $M$, then $M/N_{1} \cong N_{2}$ if and only
if $M/N_{2} \cong N_{1}$. This condition, introduced by 
G.~Ehrlich  in~\cite{Ehr}, has been investigated by W.K.~Nicholson and
E.~S\'anchez Campos in~\cite{NISC04, NISC05}, and by J.~Chen, Y.~Li
and Y.~Zhou in \cite{CLZ}. An extension to groups of results of
Ehrlich was obtained by Li and Nicholson
in~\cite{LiNic}.

The following analog  for groups of  the definition of  Ehrlich was
given  by Li, 
Nicholson and L.~Zan  in~\cite{LiNicZan}.
\begin{definition}\label{def:morphic}
 A group $G$ is said to be  \emph{morphic} if, whenever $N$ is a normal
 subgroup of $G$, such that there is  an epimorphism $\phi : G \to N$,
 then $G/N \cong \ker(\phi)$.

 In other words, $G$ is morphic if for every pair $N_{1}$ and $N_{2}$ of
 normal subgroups  of $G$, each of the conditions
 $G/N_{1} \cong N_{2}$ and $G/N_{2} \cong N_{1}$ implies the other.
\end{definition}
In~\cite{LiNicZan}  finite, nilpotent groups  were considered,  and it
was shown  that such a  group is  morphic if and  only if each  of its
Sylow  subgroups is  morphic.   This  leads to  the  study of  finite,
morphic $p$-groups.

Clearly  finite, homocyclic  $p$-groups  are morphic,  and  so is  the
nonabelian group  of order  $p^{3}$ and exponent  $p$, for $p$  an odd
prime.    F.~Aliniaeifard,   Li   and   Nicholson   conjectured that
these are the only examples:
\begin{conj}[\protect{\cite[Conjecture~3.5]{AliLiNic})}]
  \label{conj:morphic}
  The  only finite,  morphic  $p$-groups are  the abelian,  homocyclic
  $p$-groups, and the  nonabelian group of order  $p^{3}$ and exponent
  $p$, for $p$ an odd prime.
\end{conj}
They were able to prove their conjecture for
two-generated groups:
\begin{theorem}[\protect{\cite[Theorem 2.10]{AliLiNic}}]
  \label{thm:two-gens}
  Let $G$ be a finite, morphic $p$-group.

  If $G$ can be generated by two elements, then
  \begin{enumerate}
  \item either $G$ is abelian and homocyclic,
  \item or $p$ is odd, and $G$ is isomorphic to the nonabelian group of order
    $p^{3}$ and exponent $p$.
  \end{enumerate}
\end{theorem}

Aliniaeifard, Li and Nicholson also proved that a finite, morphic
$p$-group satisfies the following property:
\begin{theorem}[\protect{\cite[Proposition 2.1]{AliLiNic}}]
  \label{thm:images}
  Let $G$ be a finite, morphic $p$-group.
  \begin{enumerate}
  \item Every subgroup of $G$ is a homomorphic image of $G$.
  \item   Every homomorphic image of $G$ is isomorphic to a normal
    subgroup of $G$.
  \end{enumerate}
\end{theorem}
The latter  result shows that  a finite, morphic $p$-group  belongs to
the class of  \emph{self dual} groups, as introduced by  A. E. Spencer
\cite  {Sp},  and  studied  more  recently  by  L.  An,  J.  Ding  and
Q.  Zhang\cite {AnDiZha}.  A group  is  said to  be self  dual if  the
isomorphism  classes  of its  subgroups  and  of its  quotient  groups
coincide.

Now An, Ding and Zhang prove the following
\begin{theorem}[\protect{\cite[Corollary~7.2]{AnDiZha}}]
  Let  $G$ be  a  finite,  self dual  $p$-group.  Then  
  \begin{enumerate}
  \item either $G$ is abelian, or
  \item $p$  is odd, and $G$  is the direct product  of the nonabelian
    group of order $p^{3}$ and  exponent $p$, by an elementary abelian
    group.
  \end{enumerate}
\end{theorem}
We observe here that a proof of Conjecture~\ref{conj:morphic} can be rather straightforwardly obtained as a corollary of this result.

In this paper  we study a class of finite  $p$-groups that properly includes that of morphic $p$-groups, and is not contained in the class of 
 self dual $p$-groups:
\begin{definition}\label{def:eamorphic}
  A  finite  $p$-group  $G$  is said  to  be  \emph{elementary abelian
  morphic}, or \emph{ea}-morphic for short,  if,
  whenever $N$  is a normal subgroup  of $G$, such that  either $N$ is
  elementary abelian, or $G/N$ is elementary abelian, then there is an
  epimorphism $\phi : G \to N$, and $G/N \cong \ker(\phi)$.
\end{definition}
 
Our main result is:
\begin{theorem}\label{thm:main}
  Let $G$ be a finite,  nonabelian \emph{ea}-morphic $p$-group. Then $G$ is
  2-generated.
\end{theorem}
Since finite, morphic $p$-groups are
\emph{ea}-morphic by~Theorem~\ref{thm:images}, this result, together with
Theorem~\ref{thm:two-gens}, 
yields another proof of
Conjecture~\ref{conj:morphic}. 

In Section~\ref{sec:triples}  we introduce  a linear  algebra setting,
and  use a  counting argument  to obtain  a crucial  estimate for  the
elementary abelian quotients of the derived subgroup. This estimate is
closely  related  to  the one  in~\cite[Lemma~5.3]{AnDiZha},  but  our
treatment is self  contained, and elementary, as it  avoids the appeal
made  in~\cite[Theorem~5.2]{AnDiZha} to  a substantial  result in  the
text  of   N.~Blackburn~and  B.~Huppert~\cite[Theorem~9.8]{Hup2}.   In
Section~\ref{sec:proof}       we        complete       the       proof
of~Theorem~\ref{thm:main}, by comparing two series of normal subgroups
on the top and on the bottom of the group.

Our notation  is mainly standard.  If $H$ is  a subgroup of  the group
$K$, or $H$ is a subspace of  the vector space $K$, we write $H \ismax
K$ to indicate that $H$ is maximal in $K$.

We write $\F_{p}$ for the field with $p$ elements, $p$ a prime.

\section{Morphic triples}
\label{sec:triples}

We record the following immediate consequence of the second
isomorphism theorem, which we will use repeatedly.
\begin{lemma}\label{lemma:obvious}
  If $N$ is a normal subgroup of the group $G$, then
  \begin{equation*}
    (G/N)' \cong \frac{G'}{G' \cap N}.
  \end{equation*}
\end{lemma}

Let $G \ne \One$ be a finite, \emph{ea}-morphic $p$-group, with minimal number
of generators $d$, that is, $\Size{G/\Phi(G)} = p^{d}$.

Let $N$ be a normal subgroup of $G$ of order $p$. Then for
every maximal subgroup $M$ of $G$ one has $G/M \cong N$, and thus,
according to Definition~\ref{def:eamorphic}, also $G/N \cong M$. In
particular we obtain
\begin{lemma}[\protect{\cite[Theorem 37 (1)]{LiNicZan}}] 
\label{lemma:allmaxareiso}
  In a finite, morphic $p$-group, all maximal subgroups are isomorphic.
\end{lemma}

From now on, assume $G$ to be nonabelian, and take $N \le G'$.
Since $M \cong G/N$, Lemma~\ref{lemma:obvious} yields
\begin{equation*}
  \Size{M'}
  =
  \Size{{G'}/{N}}
  =
  \frac{\Size{G'}}{p},
\end{equation*}
so that 
\begin{equation*}
  \Size{G'/M'} = p.
\end{equation*}
Consider the characteristic subgroup $K$ of $G$ defined by
\begin{equation*}
  K = \bigcap \Set{M': M \ismax G}.
\end{equation*}
$G'/K$ is isomorphic to a subgroup of $\prod_{M \ismax G} G'/M'$,
and thus it is elementary abelian.

Consider
the map
\begin{equation*}
  \begin{aligned}
    \beta :\ &G/\Phi(G) \times G/\Phi(G) \to G'/K\\
    &(a \Phi(G), b \Phi(G)) \mapsto [a, b] K.
  \end{aligned}
\end{equation*}
$\beta$ is well defined, as we have
\begin{lemma}\label{lemma:gi-phi-kappa}
  $[G,  \Phi(G)] \le K$. In particular, $\Phi(G)' \le K$.  
\end{lemma}
\begin{proof}
  $(G/M')' = G'/M'$ has order $p$. This yields first $[G, G'] \le M'$ for
  all $M \ismax G$, so that $[G, G'] \le K$.
  Also, if $a, b \in G$, then for
  all $M \ismax G$ we have
  \begin{equation*}
    [a, b^{p}] \equiv [a, b]^{p} \equiv 1 \pmod{M'},
  \end{equation*}
  as we have just seen that $[b, [a, b]] \in M '$. Therefore also 
  $[G, G^{p}] \le K$.
\end{proof}

This also yields that $\beta$ is bilinear, as
\begin{equation*} 
[a b, c] 
=
[a, c] \cdot [[a, c], b] \cdot [b, c] 
\equiv 
[a, c] \cdot [b, c] 
\pmod{K}. 
\end{equation*}
The $\F_{p}$-vector spaces $V = G/\Phi(G)$  and $W = G'/K$, with
the map $\beta$, thus satisfy the following definition.
\begin{definition}
  \label{def:morphic-triple}
  Let $V, W$ be vector spaces over $\F_{p}$, and
  \begin{equation*}
    \begin{aligned}
      \beta :\ &V \times V \to W
             \\&(v_{1}, v_{2}) \mapsto [v_{1}, v_{2}]
    \end{aligned}
  \end{equation*}
  be an alternating bilinear map. If $U_{1}, U_{2}$ are subspaces of $V$, write
  $[U_{1}, U_{2}]$ for the linear span of $\beta(U_{1}, U_{2})$, and 
  shorten $[U_{1}, U_{1}]$ to $U_{1}'$.

  $(V, W, \beta)$ is said to be a \emph{morphic triple} if the
  following conditions hold.
  \begin{enumerate}
  \item $V' = W$.
  \item For every $U \ismax V$ one has $U' \ismax W$.
  \item $\bigcap \Set{ U' : U \ismax V } = \Zero$.
  \end{enumerate}
\end{definition}
Considering morphic triples alone is not sufficient to prove
Theorem~\ref{thm:main}, as one can construct examples of morphic
triples that are not associated to finite, morphic $p$-groups. 

\begin{proposition}\label{prop:counting_derived}
  Let $(V, W, \beta)$ be a morphic triple.

  Let $U \ismax V$. Then there exist a unique $T = \mathcal{T}(U) \ismax
  U$ which satisfies the following property.

  For $S \ismax V$, the following are equivalent.
  \begin{enumerate}
  \item $U' = S'$, and
  \item $S \ge T$.
  \end{enumerate}
\end{proposition}

\begin{proof}
  Let $a \in V \setminus U$. Consider the linear map $\tau$
  \begin{equation*}
    U \to W \to W/U'
  \end{equation*}
  given by $x \mapsto [a, x] + U'$. Since $W = V' = \Span{a, U}' = [a,
    U] + U'$, this is a surjective linear map,
  with $\dim(W/U') = 1$. Thus
  \begin{equation*}
    T 
    = 
    \mathcal{T}(U) 
    = 
    \ker(\tau) 
    = 
    \Set{x \in U : [a, x] \in U'}
  \end{equation*}
  is a maximal subspace of $U$. By definition, $[a, T] \le U'$.

  $T$ is easily seen to be independent
  of the choice of $a \in V \setminus U$.
  Thus if $T \ismax S \ismax V$, and $S \ne U$, we may assume $S =
  \Span{a, T}$. It follows that
  \begin{equation*}
    S' = [a, T] + T' \le U' + T' \le U'
  \end{equation*}
  and thus $S' = U'$, as they are both maximal subspaces of $W$.

  Suppose conversely that $U \ne S \ismax V$, and $S' = U'$. Thus $S
  \cap U \ismax U$, and if $S = \Span{a, S \cap U}$, then $a \notin
  U$, and we have  $[a, S \cap U] \le S' = U'$, so that $S \cap U =
  \mathcal{T}(U)$. 
\end{proof}

\begin{corollary}
  \label{cor:spread}
  For each $U \ismax V$, the set
  \begin{equation*}
    \Set{ S \ismax V : S' = U'}
\end{equation*}
has $p+1$  elements, namely the subspaces  $S$ such that
$\mathcal{T}(U)  \ismax S 
\ismax V$. 
\end{corollary}

The set $\mathcal{M}$ of
maximal subspaces of $W$ has 
\begin{equation*}
  1 + p + \dots + p^{e-1},
\end{equation*}
elements, where $e = \dim(W)$. Corollary~\ref{cor:spread}
implies that the subset
\begin{equation*}
  \mathcal{Z} 
  = 
  \Set{ Z \ismax W : \text{$Z = U'$ for some $U \ismax V$}}
\end{equation*}
 of  $\mathcal{M}$ has $(1 + p + \dots + p^{d-1})/(1+p)$
 elements. Thus $d$ is even, and 
\begin{equation*}
  \frac{1 + p + \dots + p^{d-1}}{1+p}
  =
  1 + p^{2} + \dots + p^{d-2}.
\end{equation*}
Clearly if $e < d-1$ we have
\begin{equation*}
  \Size{\mathcal{M}}
  =
  1 + p + \dots + p^{e-1}
  <  
  1 + p^{2} + \dots + p^{d-2} 
  =
  \Size{\mathcal{Z}},
\end{equation*}
a contradiction. We have obtained
\begin{corollary}\label{lemma:sizeofW}
  In a morphic triple $(V, W, \beta)$, we have
  \begin{equation*}
    \dim(W) \ge \dim(V) - 1.
  \end{equation*}

  In particular, in  a finite, nonabelian, \emph{ea}-morphic  $p$-group $G$ with
  minimal number of generators $d$ we have
  \begin{equation*}
    \Size{G'/K} \ge p^{d-1}.
  \end{equation*}
\end{corollary}

\section{Proofs}
\label{sec:proof}

 We are now ready to prove Theorem~\ref{thm:main}.  
 
 Let $G$  be a  finite,
nonabelian, \emph{ea}-morphic group, with minimum number of generators $d >
2$. We want to derive a contradiction. 

By  Definition~\ref{def:eamorphic}, there  is  an  epimorphism $G  \to
\Phi(G)$, and  if $E$ is its  kernel, so that $G/E  \cong \Phi(G)$, we
also have  $G/\Phi(G) \cong E$.  Since $\Size{G/\Phi(G)} =  p^{d}$, we
have that  $E$ is an  elementary abelian,  normal subgroup of  $G$, of
order $p^{d}$.  Let
\begin{equation*}\label{eq:thecap}
  p^{t} = \Size{E \cap G'},
\end{equation*}
so that $t \le d$.
Since $G/E
\cong \Phi(G)$, Lemma~\ref{lemma:obvious} yields
\begin{equation*}
  \Size{\Phi(G)'}
  =
  \Size{\frac{G'}{E \cap G'}}
  =
  \frac{\Size{G'}}{p^{t}}.
\end{equation*}
In view of  Lemma~\ref{lemma:gi-phi-kappa}~and
Corollary~\ref{lemma:sizeofW}, we obtain
\begin{equation*}\label{eq:Gprime-over-PhiGprime}
  \Size{E \cap G'}
  =
  \Size{\frac{G'}{\Phi(G)'}}
  =
  p^{t},
  \qquad
  \text{for $t \in \Set{d, d-1}$}.
\end{equation*}
Let thus $L$ be an elementary abelian, normal subgroup of order
$p^{d-1}$ contained in $G'$ (we take $L$ to be $E \cap G'$ if this has
order $p^{d-1}$, otherwise if $E \le G'$ we take $L$ to be a maximal
subgroup of $E$ which is 
normal in $G$), and let
\begin{equation*}
  \One = L_{0} \ismax L_{1} \ismax L_{2} \ismax \dots \ismax L_{d-1} =
  L \le G',
\end{equation*}
be a series of subgroups, each normal in $G$. In particular,
$\Size{L_{i}} = p^{i}$ for each $i$. 

Consider an
arbitrary series
\begin{equation*}
  \Phi(G) = S_{0} \ismax S_{1} \ismax S_{2} \ismax \dots \ismax
  S_{d-1} \ismax S_{d} = G.
\end{equation*}
(We will make a more  precise choice of $S_{2}$ later.) In particular,
each $S_{i}$ is normal in $G$, and $\Size{G/S_{i}} = p^{d-i}$ for each
$i$. Clearly  for $i \ge  1$ one has  $G/S_{i} \cong L_{d-i}$  as both
groups   are    elementary   abelian,   of    order   $p^{d-i}$.    By
Definition~\ref{def:morphic}, we also have
\begin{equation*}
  G/L_{d-i} \cong  S_{i},
  \qquad
  \text{for all $i \ge 1$.}
\end{equation*}
For $i \ge 1$ we thus have from~Lemma~\ref{lemma:obvious}
\begin{equation*}
  \Size{S_{i}'}
  =
  \Size{{G'}/{L_{d-i}}}
  =
  \frac{\Size{G'}}{p^{d-i}},
\end{equation*}
that is,
\begin{equation}\label{eq:sizeofLi}
  \Size{G' / S_{i}'} = p^{d-i}.
\end{equation}

In particular,  $\Size{G'/S_{1}'} =  p^{d-1}$. Since $S_{1}  = \Span{a,
  \Phi(G)}$ for some $a \in G$, Lemma~\ref{lemma:gi-phi-kappa} implies
$S_{1}' \le [\Span{a}, \Phi(G)] \, \Phi(G)' \le  K$, so that
Corollary~\ref{lemma:sizeofW} 
yields $\Size{G'/K} = p^{d-1}$ and $S_{1}'= K$. It follows that $S_{i}'
\ge S_{1}' = K$ for each $i \ge 1$.
We obtain from~\eqref{eq:sizeofLi} 
\begin{equation*} 
\Size{S_{i}'/K} = \Size{(G'/K)/(G'/S_{i}')}  =   p^{i-1}
\end{equation*} 
for  each  $i   \ge  1$.   

In
particular,  $\Size{S_{2}'  /  K}  =  p$  and  
$\Size{S_{3}'  /  K}  =p^{2}$. 
(This is the only  point where we use the assumption  $d > 2$.) Since
$S_{3} = \Span{a, b, c, \Phi(G)}$ for  some $a, b, c \in G$,
Lemma~\ref{lemma:gi-phi-kappa}  yields  $S_{3}' \le  \Span{[a,b], 
  [a,c], [b, c], K}$. Since every  element of the exterior square of a
vector space of dimension $3$ is  a decomposable tensor, we may choose
$a, b,  c$ so that $[a,  b] \in K$.   Choosing
$S_{2} =  \Span{a, b,  \Phi(G)}$, Lemma~\ref{lemma:gi-phi-kappa}
yields $S_{2}' \le K$, a final contradiction. 

Finally, we derive a proof of Conjecture~\ref{conj:morphic}.

If $G$ is a finite, abelian, morphic $p$-group, it is not difficult to
see from Lemma~\ref{lemma:allmaxareiso} that $G$ must be homocyclic.

In  view of  Theorem~\ref{thm:two-gens},  let thus assume that $G$ is
nonabelian.  
Since morphic $p$-groups are \emph{ea}-morphic, by
Theorem~\ref{thm:main} we conclude that $G$ is 2-generated, and by
Theorem~\ref{thm:images} we have the result.

\providecommand{\bysame}{\leavevmode\hbox to3em{\hrulefill}\thinspace}
\providecommand{\MR}{\relax\ifhmode\unskip\space\fi MR }
\providecommand{\MRhref}[2]{%
  \href{http://www.ams.org/mathscinet-getitem?mr=#1}{#2}
}
\providecommand{\href}[2]{#2}

\end{document}